\documentclass[10pt]{elsarticle}
\usepackage{amssymb,amsmath,amsthm,bm}
\usepackage[left=3cm, right=3cm, top=3.5cm, bottom=3.5cm]{geometry} 
\usepackage{amsfonts}
\usepackage{epsfig}
\usepackage{amsbsy}
\usepackage{color}
\biboptions{sort&compress}
\usepackage{subfigure}
\usepackage{graphicx}
\graphicspath{{figures/}}

\newlength\imagewidth
 \newtheorem{lemma}{Lemma}[section]
 \newtheorem{theorem}{Theorem}[section]
 
 \newtheorem{definition}{Definition}[section]
\newtheorem{remark}{Remark}[section]

\usepackage{color}
\definecolor{dgreen}{rgb}{0,.6,0}

\usepackage[normalem]{ulem}
\begin{document}

\begin{frontmatter}
\title{Modulational instability of small amplitude periodic traveling
waves in the $b$-family of Novikov equation}

\author{Xin Zhao}
\address{ Hunan Provincial Key Laboratory of Mathematical Modeling and Analysis in Engineering,\\
School of Mathematics and Statistics, Changsha University of Science and Technology,\\
Changsha, 410114,  P. R. China}
\author{Lin Lu }
\cortext[mycorrespondingauthor]{Corresponding author.}
\address{ School of Mathematics and Statistics, \\ Hunan First Normal University,
Changsha, 410205,  P. R. China}
\author{Aiyong Chen \corref{mycorrespondingauthor}}
\ead{aiyongchen@163.com}
\address{ School of Mathematics and Statistics, \\ Hunan First Normal University,
Changsha, 410205,  P. R. China}

\begin{abstract}
We study the modulational instability of smooth, small-amplitude periodic traveling wave solutions to the $b$-family of Novikov equation with cubic nonlinearity with an arbitrary
coefficient $b>0$. Our approach is based on applying spectral perturbation theory to the corresponding linearization process. We derive a modulation instability index dependent on the nonlinear parameter $b$ and the fundamental wave number, and prove that when this index is negative, sufficiently small periodic traveling waves in the Novikov equation $b$-family exhibit spectral instability to long-wavelength perturbations. This confirms the well-known Benjamin-Feir instability in the $b$-family of Novikov equation. 
\end{abstract}


\end{frontmatter}

\section{Introduction}
\numberwithin{equation}{section}

In this paper, we consider the $b$-family of Novikov ($b$-Novikov) equation \cite{b-novikov2013}
\begin{equation}\label{y1.3}
u_t-u_{xxt}=buu_xu_{xx}-(b+1)u^2u_x+u^2u_{xxx},\quad t>0, ~x\in\mathbb{R},
\end{equation}
where $b>0$. Note that when setting $b=3$, equation (\ref{y1.3}) becomes the following well-known Novikov equation \cite{novikov2009}
\begin{equation}\label{y1.2}
u_t-u_{xxt}=3uu_xu_{xx}-4u^2u_x+u^2u_{xxx},
\end{equation}
which was introduced by Novikov in a symmetry classification of nonlocal partial differential equations with cubic nonlinearity. Equation 
(\ref{y1.2}) can be regarded as a generalization to a cubic nonlinearity of the Camassa-Holm (CH) equation \cite{ch1981,ch1993}
\begin{equation}\label{ch}
u_t-u_{txx}=2u_xu_{xx}-3uu_x+uu_{xxx}
\end{equation}
and the Degasperis-Procesi (DP) equation \cite{dp1999}
\begin{equation}\label{dp}
u_t-u_{txx}=3u_xu_{xx}-4uu_x+uu_{xxx}.
\end{equation}
Constantin et al. \cite{Constantin2001PRS,Constantin2006IP,Constantin2009ARMA,Constantin1998AM,Constantin1998ASNSP}
 have conducted extensive research on the CH equation, covering  its scattering problem, inverse scattering transform,
  fluid dynamics correlation with the DP equation, as well as wave breaking, global existence and blow-up problems  of the equation.

All the three equations (\ref{y1.2})-(\ref{dp}) share many common properties. For instance, they are all completely integrable in the sense that they all have a Lax pair representation, a bi-Hamiltonian structure, and an infinite sequence of conservation laws. 
 Hone et al. \cite{Hone etal2009} calculated the explicit formulas for multi-peakon solutions of (\ref{y1.2}). Himonas and Holliman \cite{A. Himonas and C. Holliman} obtained that equation (\ref{y1.2}) is well-posed in Sobolev spaces $\mathrm{H}^s$ with $s>\frac{3}{2}$ on both the line and the circle with continuous dependence on initial data. In \cite{Grayshan2013,yan etal2013} the authors proved that the data-to-solution map is not globally uniformly continuous on $\mathrm{H}^s$ for $s<\frac{3}{2}$, this result supplements Himonas and Holliman’s works. Tiglay \cite{Tiglay2011} showed the local well-posedness of the problem in Sobolev spaces and the existence and uniqueness of solutions for all time using orbit invariants. For analytic initial data, the existence and uniqueness of analytic solutions for equation (\ref{y1.2}) were also obtained in \cite{Tiglay2011}. Analogous to the CH equation, the Novikov equation possesses a blow-up phenomenon \cite{Jiang2012} and global weak solutions \cite{wu2011}.  Palacios \cite{Palacios2020} proved the asymptotic stability of peakon solutions under perturbations satisfying that their associated momentum density defines a non-negative Radon measure. 

At the same time, equation (\ref{y1.3}) is a particular case (for
$k = 2$) of the generalized Camassa-Holm (gCH) equation \cite{gkbch2013}
\begin{equation}\label{gkbch}
u_{t}-u_{xxt}=u^{k}u_{xxx}+bu^{k-1}u_{x}u_{xx}-(b+1)u^{k}u_{x},\quad u=u(x,t).
\end{equation}
In \cite{gkbch2013}, it was proved among other things, that each member of the family of equations (\ref{gkbch}) 
possesses peakon travelling wave solutions, which on the line, have the form
\begin{equation}
u(x,t)=c^{\frac{1}{k}}e^{-|x-ct|},
\end{equation}
where $c\neq0$ is the velocity of the wave. Obviously, (\ref{y1.3}) admits the peakon solution
\begin{equation}
\tilde{u}(x,t)=\sqrt{c}e^{-|x-ct|}.
\end{equation}
  The local well posedness to the Cauchy problem of (\ref{y1.3}) in critical Besov spaces was studied in \cite{Zhou S2013}. Himonas and Holliman \cite{Himonas AA2022} analyzed the instability and non-uniqueness of the initial value problem for equation (\ref{y1.3}) on lines and circles when initial data belong to the Sobolev space $\mathrm{H}^s$ with $s < 3/2$. It proved that equation (\ref{y1.3}) is ill-posed for $s < 3/2$ when $b > 2$. This result was achieved by constructing a bimodal solution with arbitrarily small initial data that collides within an arbitrarily short time. In \cite{da Silva PL2015}, a complete group classification was carried out. From the Lie symmetry generators, the authors obtained the exact solutions and a nontrivial conservation law for equation (\ref{y1.3}). In \cite{Efstathiou2022}, analytic approximations of the peakon solution for equation (\ref{y1.3}) were obtained by applying the homotopy analysis method.  Furthermore, Deng  and Lafortune \cite{dengxijun2025} proved spectral and linear instability on $L^2(\mathbb{R})$ of peakons for equation (\ref{y1.3}).

Another equation strongly connected with (\ref{y1.3}) is the $b$-family equation  \cite{fanlili2025} 
\begin{equation}\label{b-eq}
u_{t}-u_{xxt}=uu_{xxx}+bu_{x}u_{xx}-(b+1)uu_{x},\quad u=u(x,t),
\end{equation}
which becomes the well-known CH equation when $b = 2$ and the DP equation when $b = 3$. In addition,  (\ref{b-eq}) can be deduced from (\ref{gkbch}) for $k = 1$. It is easy to see that (\ref{y1.3})  has nonlinear terms that are cubic, rather than quadratic of $b$-family equation. Equation (\ref{y1.3}) only differs from  equation (\ref{b-eq}) by a multiplying factor $u$ applied to the RHS. Recently, Fan  et al. \cite{fanlili2025} proved the modulational instability for the periodic traveling waves of equation (\ref{b-eq}), where the instability follows
from the presence of the spectrum of a linearized operator in the right half
plane of the complex plane.

Ehrman et al. \cite{novikov2025} studied the spectral stability of smooth, small-amplitude periodic traveling wave solutions to the Novikov equation.
 Their primary approach was to focus on the neighborhood of the origin in the spectral plane, conduct an in-depth analysis of
  the $L^2(\mathbb{R})$-spectral properties of the associated linearized operator, and use spectral perturbation theory to conclude that these small-amplitude periodic solutions exhibit spectral instability to long-wavelength perturbations when the wave number exceeds a critical value. 
  Conversely, the periodic solutions are spectrally stable when the wave number is below this critical threshold. 
   This study established a clear classification criterion for the spectral stability of such solutions, 
   thereby enriching the relevant results concerning small-amplitude periodic waves in the context of modulational instability for the Novikov equation. 
    Johnson and Oregero \cite{26} studied the nonlinear wave modulation phenomenon of periodic traveling wave solutions with arbitrary
     amplitude in the CH equation. Using the nonlinear WKB/multiple scales expansion method, they derived the Whitham modulation system for 
     the CH equation and established a rigorous connection between this system and the spectral stability of periodic wave trains with respect
      to localized perturbations.

Obviously, as a classical Camassa-Holm-type equation with cubic nonlinearity, the modulational instability of small-amplitude periodic traveling waves in the Novikov equation has been investigated. However, there are currently no relevant studies on the modulational instability of small-amplitude periodic traveling waves in the $b$-family of Novikov equation. This type of equation possesses both the cubic nonlinear characteristics of the Novikov equation and the parameter adjustability of the  $b$-family equation, and the laws governing the modulational stability of its periodic traveling waves as well as the synergistic influence mechanism between the parameter $b$ and the wave number remain to be further explored. 
Therefore, inspired by the aforementioned work, this paper focuses on the modulational instability of small-amplitude periodic 
traveling waves in equation (\ref{y1.3}), with the aim of linking the research on modulational instability 
 in the Novikov equation to the nonlinear parameter $b$ of the $b$-family equation, thereby further enriching the relevant research on modulational 
instability of Camassa-Holm-type equation and providing some ideas for the stability control of nonlinear wave system.

Modulational instability (also known as Benjamin-Feir instability), a classical stability problem of periodic traveling waves under long-wave perturbations, can be traced back to the 1960s, when Benjamin and Feir \cite{16,17} discovered that small-amplitude spatially periodic Stokes waves are unstable under long-wave perturbation. In 1995, Bridges and Mielke \cite{18} provided the first
  rigorous proof of the modulational instability of Stokes waves in the case of finite depth. In 2020, Nguyen and Strauss \cite{19} completed the proof for the infinite depth case. Berti et al. have further clarified the "figure-8" distribution of unstable eigenvalues in the cases of deep water \cite{20}, finite depth \cite{21} and critical depth \cite{22}. In addition to Stokes waves,
   modulational instability also occurs in various models, including quasi-periodic solutions and periodic standing waves of the nonlinear Schrödinger (NLS) equation \cite{23,24} and periodic standing waves of the generalized Korteweg-de Vries equation \cite{25}. Many scholars have conducted in-depth research on the modulation instability of periodic traveling wave solutions in several important models, such as Whitham equation \cite{27}, BBM and regularized Boussinesq-type equations \cite{28}, generalized BBM equation \cite{29}, 
   full-dispersion CH equation \cite{31}
   and KdV-type equations \cite{33}. Recent studies on the modulational instability of the generalized KdV equation \cite{34} have also yielded results similar 
 to those of Berti et al., namely that the unstable eigenvalues near zero exhibit a "figure-8" pattern.

To analyze modulational  instability of small amplitude periodic traveling waves of equation (\ref{y1.3}), we introduce some notational conventions and notations in the following. The space 
\(L^2(\mathbb{R})\) consists of all real or complex-valued, and Lebesgue-measurable function \(f \) on \(\mathbb{R}\) satisfying
\[\|f\|_{L^2(\mathbb{R})}=\left(\int_{\mathbb{R}}|f(x)|^2\mathrm{d}x\right)^{1/2}<+\infty.\]
Similarly, \(L^2(\mathbb{T})\) denotes the space of \(2\pi\)-periodic, real or complex-valued, measurable function  \(f \) on \(\mathbb{R}\) satisfying
\[\|f\|_{L^{2}(\mathbb{T})}=\left(\frac{1}{2\pi}\int_{-\pi}^{\pi}|f(x)|^{2}\mathrm{d}x\right)^{1/2}<+\infty\]
and
\(\| f \|_{L^\infty(\mathbb{T})} := \mathrm{ess\,sup}_{-\pi < z \leq \pi} |f(z)| < \infty\). Let \(H^1(\mathbb{T})\) consists of \(L^2(\mathbb{T})\) function whose derivative is in 
\(L^2(\mathbb{T})\). Let \(H^\infty(\mathbb{T}) = \bigcap_{k=0}^\infty H^k(\mathbb{T})\). 
For \(f \in L^1(\mathbb{T})\), the Fourier series of \(f \)  is defined as
$$\sum_{n\in\mathbb{Z}}\widehat{f}_{n}\mathrm{e}^{\mathrm{i}nz},$$
where  $\widehat{f}_{n}=\frac{1}{2\pi}\int_{-\pi}^{\pi}f(z)\mathrm{e}^{-\mathrm{i}nz}\mathrm{d}z.$
For \(f \in L^2(\mathbb{T})\), its Fourier series converges  to  \(f \) pointwise almost everywhere. The \(L^2(\mathbb{T})\)-inner product is given by
\[\langle f,g\rangle=\frac{1}{2\pi}\int_{-\pi}^{\pi}f(z)\bar{g}(z)\mathrm{d}z=\sum_{n\in\mathbb{Z}}\widehat{f}_{n}\overline{\widehat{g}}_{n}.\]
Finally, for operators $F$ and $G$ on a Hilbert space, their commutator is defined as\[[F,G]:=F\circ G-G\circ F.\]
The remainder of this paper is organized as follows. In Section 2, we employ the implicit function theorem and the Lyapunov-Schmidt reduction method to verify the existence of one-dimensional small-amplitude periodic traveling waves for the b-Novikov equation, whose parametric form is presented in Lemma 2.1. In Section 3, we perform a linearization of equation (\ref{y1.3}) around the acquired periodic travelling waves. Then we formulate the spectral stability problem and deliver the results of Bloch wave decomposition,  and we define modulational stability and modulational instability rigorously. In Section 4, we determine the spectrum of the unperturbed operator. In Section 5, we employ spectral perturbation theory to project the infinite-dimensional spectral problem onto the three-dimensional critical feature space related to modulational stability (or instability) under the small-amplitude limit, providing the proof of Theorem 5.1.

\section{Existence of small-amplitude periodic traveling wave solutions}
Taking traveling wave transformation \( u(x,t) = \varphi(x-ct) \), where \( c > 0 \) is the wave speed, then equation (\ref{y1.3}) becomes
\begin{equation}\label{2.1}
-c\varphi' + c\varphi''' = b\varphi\varphi'\varphi'' -(b+1)\varphi^2\varphi' +\varphi^2\varphi''' ,
\end{equation}
which is equivalent to
\begin{equation}\label{2.2}
(\varphi^2-c)(\varphi-\varphi^{\prime\prime})^{\prime}+b\varphi\varphi^{\prime}(\varphi-\varphi^{\prime\prime})=0.
\end{equation}
According to the fundamental theory of ordinary differential equations (ODEs), when \( \varphi^2(x) < c \) or \( \varphi^2(x) > c \)  holds for all \(x\in\mathbb{R}\), it necessarily follows that \(\varphi\in C^{\infty}(\mathbb{R})\). This paper restricts its focus to wave solutions satisfying
\begin{equation}\label{tj1}
\varphi^{2}(x)<c\text{ for all }x\in\mathbb{R}.
\end{equation}
 By multiplying (\ref{2.2}) by the integrating factor \((\varphi-\varphi^{\prime\prime})^{\frac{2-b}{b}}\), the equation (\ref{2.2}) can be expressed in  conservation form
\begin{equation}\label{y2.4}
\frac{d}{dx}\left(\left(\varphi-\varphi^{\prime\prime}\right)^{2/b}\left(c-\varphi^2\right)\right)=0.
\end{equation}
Additionally, it is assumed that
\begin{equation}\label{tj2}
\varphi-\varphi''>0 \text{ for all }x\in\mathbb{R},
\end{equation}
equation (\ref{y2.4}) is equivalent to
\begin{equation}\label{shs1}
\varphi-\varphi^{\prime\prime}=\frac{d}{(c-\varphi^2)^{b/2}},
\end{equation}
where  \(d > 0\) is a constant of integration. Multiplying by $\varphi'$ and integrating, equation (\ref{shs1}) can be transformed into the quadrature representation
\begin{equation}\label{mjcl}
\frac{1}{2}(\varphi^{\prime})^{2}=E+\frac{1}{2}\varphi^{2}-\int_{0}^{\varphi}\frac{d}{(c-\varphi^{2})^{\frac{b}{2}}} \mathrm{d} \varphi,
\end{equation}
where $E\in\mathbb{R}$ is another constant of integration. 
When $b=2,$ equation (\ref{mjcl}) can be explicitly written as
\begin{equation*}
\frac{1}{2}(\varphi^{\prime})^{2} - \frac{1}{2}\varphi^{2} + \frac{ d }{\sqrt{ c }} \mathrm{arctanh} (\frac{\varphi}{\sqrt{c}})=E.
\end{equation*}
When $b=3,$ equation (\ref{mjcl}) becomes
\begin{equation*}
\frac{1}{2}(\varphi^{\prime})^{2} - \frac{1}{2}\varphi^{2} + \frac{ d \varphi }{c \sqrt{ c- \varphi^2}} = E.
\end{equation*}
The periodic orbits determined by equation (\ref{shs1}) are shown in Figure \ref{phase-figure}.
  \begin{figure}[htbp]
    \centering
    \includegraphics[width=0.45\textwidth]{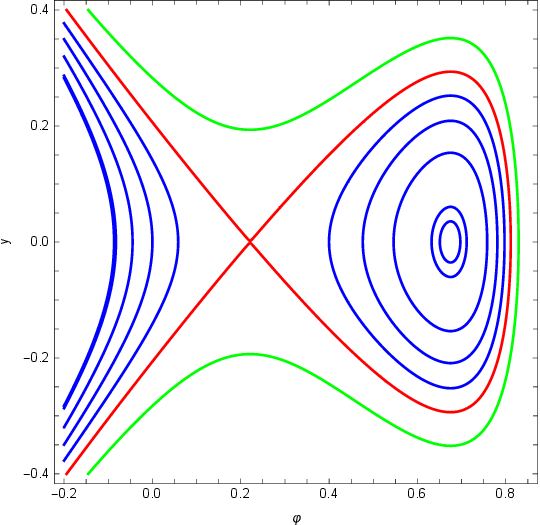}
    \caption {The periodic orbits determined by equation  (\ref{shs1}).}
    \label{phase-figure}
\end{figure}
According to the elementary phase plane analysis, as long as the effective potential function
\begin{equation}
V(\varphi;d,c)=\int_{0}^{\varphi}\frac{d}{(c-\varphi^{2})^{\frac{b}{2}}}\mathrm{d} \varphi-\frac{1}{2}\varphi^{2}
\end{equation}
possesses a strict local minimum within the interval $\varphi\in(-\sqrt{c},\sqrt{c})$, equation (\ref{2.1}) possesses a smooth periodic solution satisfying (\ref{tj1}) and (\ref{tj2}).
In fact, $$V^{\prime}(\varphi;d,c)= \frac{d}{(c-\varphi^2)^\frac{b}{2}} - \varphi,$$
where the symbol prime  represents taking the derivative of $\varphi$.
By direct calculation, we obtain 
\begin{equation}\label{V-Vprime-limV}
V(0;d,c)=0,~~~V^{\prime}(0;d,c)=\frac{d}{c^{\frac{b}{2}}}>0,
~~~\lim_{\varphi \rightarrow (\sqrt{c})^-} V (\varphi;d,c) = +\infty.
\end{equation}
Let
\begin{equation*}
D(\varphi;d,c):=d-\varphi(c-\varphi^{2})^{\frac{b}{2}}.
\end{equation*}
Then solving the equation 
$$\begin{aligned}
D^{\prime}(\varphi;d,c) =(c-\varphi^2)^{\frac{b-2}{2}}\big[(b+1)\varphi^{2}-c\big]=0
\end{aligned}$$ yields 
 $\varphi=\sqrt{\frac{c}{b+1}}$.
When $d>0$ and $c>(b+1)b^{-\frac{b}{b+1}}d^{\frac{2}{b+1}},$ we have
\begin{equation}\label{Vprimete}
V^{\prime}(\sqrt{\frac{c}{b+1}};d,c)=d (\frac{b+1}{bc})^\frac{b}{2} - \sqrt{\frac{c}{b+1}}<0.
\end{equation}
Moreover, since $$V^{\prime\prime}(\varphi;d,c)= \frac{d b \varphi }{(c-\varphi^2)^\frac{b+2}{2}} - 1,$$
we have
\begin{equation}\label{Vprimeprime1}
V^{\prime\prime}(0;d,c)= - 1 < 0,
~~~~~~~~~
 \lim_{\varphi \rightarrow (\sqrt{c})^-} V^{\prime\prime} (\varphi;d,c) = +\infty.
\end{equation}
Since the inflection points of $V$ are determined by the equation $V^{\prime\prime} (\varphi;d,c) = 0,$ 
that is
\begin{equation}\label{Vinflectionpoint}
 d^2 b^2 \varphi^2 = (c - \varphi^2)^{b+2},
 \end{equation}
 by using (\ref{V-Vprime-limV}), (\ref{Vprimete}), (\ref{Vprimeprime1}) and (\ref{Vinflectionpoint}),
  we obtain that there exists a unique inflection point $\varphi_0 \in (\sqrt{\frac{c}{b+1}},\sqrt{c})$ 
  such that 
$$V^{\prime\prime}(\varphi;d,c)<0~{\mathrm{for~all}}~|\varphi|<{\varphi_0},~~~ V^{\prime\prime}(\varphi;d,c)>0~\mathrm{for~all}~{\varphi_0}<\varphi<\sqrt{c}.$$
Therefore, the effective potential function possesses a non-degenerate strict local minimum at a certain point $w_{0}\in(\varphi_0,\sqrt{c})$. 
By elementary phase plane analysis, we can conclude that for every pair of positive parameters $d>0$ and $c>(b+1)b^{-\frac{b}{b+1}}d^{\frac{2}{b+1}}$,
 there exists a family of periodic solutions of (\ref{shs1}) that oscillate around the equilibrium solution $w_0$.
 The schematic diagram of the  function $V (\varphi; d, c)$ is shown in Figure \ref{Vfunction}.
\begin{figure}[htbp]
    \centering
    \includegraphics[width=0.45\textwidth]{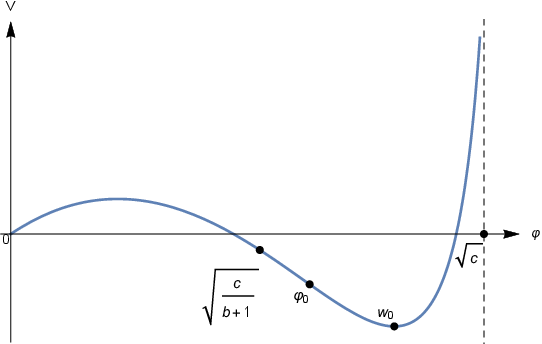}
    \caption {A schematic diagram of the potential function $V (\varphi; d, c)$ for  $d>0$ and $c>(b+1)b^{-\frac{b}{b+1}}d^{\frac{2}{b+1}}$.
   The horizontal coordinates of each black dot are marked in the graph.}
    \label{Vfunction}
\end{figure}

Let \( \varphi \) be a \( 2\pi/k \)-periodic function of its variable, for some wave number \( k > 0 \) . 
We introduce the scaling \( z = kx \) so that \( w(z) := \varphi(x) \) becomes a \( 2\pi \)-periodic function satisfying
\begin{equation}\label{y2.2}
-cw'+ck^2w'''=bk^2ww'w''-(b+1)w^2w'+k^2w^2w'''
\end{equation}
or, equivalently,
\begin{equation}\label{y-2.3}
(w-k^2w^{\prime\prime})(c-w^2)^{b/2}=d.
\end{equation}
Define  \( F: H^2(\mathbb{T}) \times \mathbb{R}_+ \times \mathbb{R}_+ \times \mathbb{R}_+ \to L^2(\mathbb{T}) \) by
\begin{equation*}
F(w;k,d,c)=(w-k^2w^{\prime\prime})(c-w^2)^{b/2}-d,
\end{equation*}
then solutions to equation (\ref{y2.2})  correspond to solutions \( w\) of
\begin{equation}\label{ODE2.4}
F(w;k,d,c)=0.
\end{equation}
Note that equation (\ref{ODE2.4}) remains invariant under the transformation \( z \to -z \) , \( z \to z+z_0 \) for all \( z_0\in\mathbb{R} \). Therefore, we assume that \( w \) is an even function with respect to \( z \). By the implicit function theorem, for some \( c = c_0 \), as long as the operator
\begin{equation*}
\partial_w F(w_0; k, d, c) =1-\frac{bw_0^2}{c-w_0^2} -k^2 \partial_z^2
\end{equation*}
fails to be an isomorphism from \( H^2(\mathbb{T}) \) to \( L^2(\mathbb{T}) \), then non-constant solutions of (\ref{ODE2.4}) may bifurcate from \( w = w_0 \).   We note that
\[
\partial_w F(w_0; k, d, c) \cos(nz) =\bigg(1-\frac{bw_0^2}{c-w_0^2} +k^2 n^2\bigg) \cos(nz).
\]
Thus, when
\begin{equation}\label{y-2.6}
c=c_0=\bigg(\frac{k^2+b+1}{k^2+1}\bigg)w_0^2,
\end{equation}
we have \( \cos(z) \in \ker\left(\partial_w F(w_0; k, d, c)\right) \).
Moreover, since the function
\[
n \mapsto 1-\frac{bw_0^2}{c-w_0^2} +k^2 n^2 \in \mathbb{R} \quad \text{for all } n \in \mathbb{N}
\]
is strictly increasing in \( n \), it follows that
\[ \ker\left(\partial_w F(w_0; k, d, c)\right) = \text{span}\{\cos(z)\}. \]
Also, the equilibrium solution $w_0$ satisfies
\begin{equation}\label{2-w0}
\begin{pmatrix}
c-w_0^2
\end{pmatrix}^{b/2}w_0=d,
\end{equation}
substituting $c = c_0$ yields a closed-form expression
\begin{equation}\label{y2.7}
w_0=d^{\frac{1}{b+1}}\bigg(\frac{b}{1+k^2}\bigg)^{-\frac{b}{2(b+1)}}
\end{equation}
and
\begin{equation}\label{y2.8}
c_0=d^{\frac{1}{b+1}}\bigg(\frac{b}{1+k^2}\bigg)^{-\frac{b}{2(b+1)}}\bigg(\frac{k^2+b+1}{k^2+1}\bigg).
\end{equation}
 Using the Lyapunov-Schmidt method, we construct a one-parameter family of non-constant, even and smooth solutions to (\ref{2.1}) near \( w = w_0(k,d) \) and \( c = c_0(k,d)\). Their small-amplitude expansions are given below.

\begin{lemma}
For each \( k > 0 \), \( d> 0 \), there exists a family of small-amplitude \( 2\pi/k \)-periodic traveling wave solutions to (\ref{y1.3}) of the form
\[
u(x,t;a,d) = w\bigl(k(x - c(k,a)t);a,d,k\bigr) 
\]
for $|a|\ll1,$ where \( w \) and \( c \) depend analytically on \( k \) and \( a \). The function \( w \) is a smooth, \( 2\pi \)-periodic and even function with respect to \( z \), and \( c \) is an even function with respect to \( a \). Moreover, as \( a \to 0 \), the following asymptotic expansion holds
\begin{equation}\label{y2.9}
w(z;a,d,k)=w_0(k,d)+a\cos z+a^2\left(e_1+e_2\cos2z\right)+O\left(a^3\right),
\end{equation}
\begin{equation}\label{y2.10}
c(a,d,k)=c_0(k,d)+a^2c_2+O\left(a^4\right),
\end{equation}
with
\begin{equation}\label{y2.11}
e_{1}=\frac{\big[-2(1+k^{2})^{2}+(-4-6k^{2}+k^{4})b+(-2-2k^{2}+k^{4})b^{2}\big]\big(\frac{b}{1+k^{2}}\big)^{\frac{-(2+b)}{2(1+b)}}}{24(1+b)d^\frac{1}{1+b}k^{2}},
\end{equation}
\begin{equation}
e_{2}=\frac{\big[(2+2b)+(4+3b)k^{2}+(2+b)k^{4}\big]\big(\frac{b}{1+k^{2}}\big)^{\frac{b}{2(1+b)}}}{12bd^{\frac{1}{1+b}}k^{2}},
\end{equation}
\begin{equation}\label{y2.12}
c_{2}=\frac{(10+8k^{2}-2k^{4})+(20+12k^{2}+k^{4})b+(10+4k^{2}+k^{4})b^{2}}{12b(1+b)},
\end{equation}
and \( w_0 \) and \( c_0 \) are given by (\ref{y2.7}) and (\ref{y2.8}) respectively. 
\end{lemma}

\begin{proof}
We discuss the small-amplitude expansion of equation (\ref{2.1}). Since \( w \) and \( c \) are analytic functions of \( a \) when \( |a| \) is small enough, and \( c \) is even in \( a \), then \( w \) and \( c \) can be written as
\begin{equation}\label{w2.13}
w(z;a,d,k) = w_0(k, d) + a \cos z + a^2 w_2(z) + a^3 w_3(z) + O(a^4)
\end{equation}
and
\begin{equation}\label{c2.14}
c(a,d,k) = c_0(k, d) + a^2 c_2 + O(a^4)
\end{equation}
as \( a \to 0 \). Here, \( w_2, w_3 \) are even and \(2\pi\)-periodic in $z$. Substituting (\ref{w2.13}) and (\ref{c2.14}) 
 into (\ref{y-2.3}) .
  At the order  $a^2$,  we get
\begin{equation*}\begin{aligned}
&-\frac{1}{2b}d^{-\frac{1}{1+b}}\bigg(\frac{bd^{\frac{2}{1+b}}\big(\frac{b}{1+k^{2}}\big)^{-\frac{b}{1+b}}}{1+k^{2}}\bigg)^{\frac{b}{2}}\bigg[-bc_{2}(1+k^{2})\bigg(\frac{b}{1+k^{2}}\bigg)^{\frac{b}{2(1+b)}}\\
&+\big(2+2b+4k^2+3bk^2+2k^4+bk^4\big)\bigg(\frac{b}{1+k^{2}}\bigg)^{\frac{b}{2(1+b)}}\cos^{2}z+2bd^{\frac{1}{1+b}}k^{2}\big(w_{2}+w_{2}^{''}\big)\bigg]=0,
\end{aligned}\end{equation*}
which is equivalent to
\begin{equation}\begin{aligned}\label{2-w2}
2bd^{\frac{1}{1+b}}k^{2}(w_{2}+w_{2}^{''})&= bc_{2}(1+k^{2})\bigg(\frac{b}{1+k^{2}}\bigg)^{\frac{b}{2(1+b)}}\\
&\quad -\bigg[(2+2b)+(4+3b)k^2+(2+b)k^4\bigg]\bigg(\frac{b}{1+k^{2}}\bigg)^{\frac{b}{2(1+b)}}\cos^{2}z\\
& =bc_{2}(1+k^{2})\bigg(\frac{b}{1+k^{2}}\bigg)^{\frac{b}{2(1+b)}}\\
&\quad -\frac{1}{2}\bigg[(2+2b)+(4+3b)k^2+(2+b)k^4\bigg]\bigg(\frac{b}{1+k^{2}}\bigg)^{\frac{b}{2(1+b)}}(1+\cos2z).
\end{aligned}\end{equation}
We assume that
$$w_2=e_1+e_2\cos2z,$$
combining this with the equation (\ref{2-w2}) yields 
  $$\begin{cases}
2b d^{\frac{1}{1+b}}k^{2}e_{1}=bc_{2}(1+k^{2})\big(\frac{b}{1+k^{2}}\big)^{\frac{b}{2(1+b)}}-\frac{1}{2}\big[(2+2b)+(4+3b)k^{2}+(2+b)k^{4}\big]\big(\frac{b}{1+k^{2}}\big)^{\frac{b}{2(1+b)}} \\
-6bd^{\frac{1}{1+b}}k^{2}e_{2}=-\frac{1}{2}\big[(2+2b)+(4+3b)k^{2}+(2+b)k^{4}\big]\big(\frac{b}{1+k^{2}}\big)^{\frac{b}{2(1+b)}} & 
\end{cases},$$
we get
$$e_{2}=\frac{\big[(2+2b)+(4+3b)k^{2}+(2+b)k^{4}\big]\big(\frac{b}{1+k^{2}}\big)^{\frac{b}{2(1+b)}}}{12bd^\frac{1}{1+b}k^{2}}$$
and
$$e_{1}=\frac{c_2(1+k^{2})(\frac{b}{1+k^{2}})^{\frac{b}{2(1+b)}}}{2d^{\frac{1}{1+b}}k^{2}}-\frac{[(2+2b)+(4+3b)k^{2}+(2+b)k^{4}](\frac{b}{1+k^{2}})^{\frac{b}{2(1+b)}}}{4bd^{\frac{1}{1+b}}k^{2}}.$$
 At the order  $a^3$,
\begin{equation*}
\begin{aligned}
&\frac{1}{3b^2} d^{-\frac{2}{1+b}} \left[ d^{\frac{2}{1+b}} \left( \frac{b}{1+k^2} \right)^{\frac{1}{1+b}} \right]^{\frac{b}{2}} \Bigg\{ (3b + 6bk^2 + 3bk^4) c_2 \left( \frac{b}{1+k^2} \right)^{\frac{b}{1+b}} \cos z \\
&+ \left[ (-4 - 3b + b^2) + (-12 - 6b + 3b^2)k^2 + (-12 - 3b + 3b^2)k^4 + (-4 + b^2)k^6 \right] \left( \frac{b}{1+k^2} \right)^{\frac{b}{1+b}} \cos^3 z \\
&- 6b d^{\frac{1}{1+b}} \left( \frac{b}{1+k^2} \right)^{\frac{b}{2(1+b)}} (1+k^2)(1+b+k^2) \cos z \, w_2+ 3b^2 d^{\frac{1}{1+b}} k^2 (1+k^2) \left( \frac{b}{1+k^2} \right)^{\frac{b}{2(1+b)}} \cos z \, w_2'' \\
&- 3b^2 d^{\frac{2}{1+b}} k^2 (w_3 + w_3'') \Bigg\} = 0,
\end{aligned}
\end{equation*}
which is equivalent to
\begin{equation}
\begin{aligned}\label{2-w3}
& \big(3b + 6bk^2 + 3bk^4\big) c_2 \left( \frac{b}{1+k^2} \right)^{\frac{b}{1+b}} \cos z+ \bigg[ (-4 - 3b + b^2) + (-12 - 6b + 3b^2)k^2  \\
&+ (-12 - 3b + 3b^2)k^4 + (-4 + b^2)k^6 \bigg] \left( \frac{b}{1+k^2} \right)^{\frac{b}{1+b}}\bigg (\frac{1}{4}\cos3z+\frac{3}{4}\cos z\bigg)\\
& - 6b d^{\frac{1}{1+b}} \left( \frac{b}{1+k^2} \right)^{\frac{b}{2(1+b)}} (1+k^2)(1+b+k^2) \cos z (e_1+e_2\cos2z) \\
& -12e_2b^2 d^{\frac{1}{1+b}} k^2 (1+k^2) \left( \frac{b}{1+k^2} \right)^{\frac{b}{2(1+b)}} \cos z \cos2z = 3b^2 d^{\frac{2}{1+b}} k^2 (w_3 + w_3'')  .\end{aligned}
\end{equation}
Noting that
\[\cos3\alpha=2\cos2\alpha\cos\alpha-\cos\alpha,\]
combining this with the equation (\ref{2-w3}) yields 
\begin{equation*}\begin{aligned}
&\frac{(2+2b)+(4+3b)k^{2}+(2+b)k^{4}}{k^{2}}(1+k^{2})\bigg(\frac{1+b+k^{2}}{2}+bk^{2}\bigg)\\ & =2(3b+6bk^{2}+3bk^{4})c_{2}-6b(1+k^{2})^{2}(1+b+k^{2})\frac{c_{2}}{k^{2}} \\
 & \quad+\frac{3}{2}\big[(-4-3b+b^{2})+(-12-6b+3b^{2})k^{2}+(-12-3b+3b^{2})k^{4}+(-4+b^{2})k^{6}\big] \\
 & \quad +3(1+k^{2})(1+b+k^{2})\frac{(2+2b)+(4+3b)k^{2}+(2+b)k^{4}}{k^{2}},
\end{aligned}\end{equation*}
thus, we have
$$c_{2}=\frac{(10+8k^{2}-2k^{4})+(20+12k^{2}+k^{4})b+(10+4k^{2}+k^{4})b^{2}}{12b(1+b)}$$
and
$$e_{1}=\frac{\big[-2(1+k^{2})^{2}+(-4-6k^{2}+k^{4})b+(-2-2k^{2}+k^{4})b^{2}\big]\big(\frac{b}{1+k^{2}}\big)^{-\frac{(2+b)}{2(1+b)}}}{24(1+b)d^\frac{1}{1+b}k^{2}}.$$
\end{proof}

\section{Spectral stability and spectral decomposition}

In this section, let \( w(z;a,d,k) \) with \(d,k>0\) and \(|a| \ll 1\) be a small-amplitude, $2\pi$-periodic travelling wave solution to the equation (\ref{y1.3}), whose existence was guaranteed by Lemma 2.1. Linearising  equations (\ref{y1.3}) at its one-dimensional periodic travelling wave solution \( w \) (given by (\ref{y2.9})), and considering the perturbation term \( w + \varepsilon V(z,t) \), yields the equation 
\begin{equation}
V_t=k\left(1-k^2\partial_z^2\right)^{-1}\mathcal{L}[w]V,
\end{equation}
where
\begin{equation}\begin{aligned}
\mathcal{L}[w] & :=c\partial_{z}-ck^{2}\partial_{z}^{3}-2(1+b)ww_{z}-(1+b)w^{2}\partial_{z} \\
 & +bk^2(w_zw_{zz}+ww_{zz}\partial_z+ww_z\partial_z^2)+k^2\left(w^2\partial_z^3+2ww_{zzz}\right).
\end{aligned}\end{equation}
Setting \( V(z, t) = e^{\lambda t} v(z) \) with \( \lambda \in \mathbb{C} \), we derive the spectral problem
\begin{equation}\label{y3.1}
\begin{aligned}
\lambda v & =k\left(1-k^2\partial_z^2\right)^{-1}\mathcal{L}[w] v:=P[w]v .
\end{aligned}
\end{equation}

\begin{definition} 
(Spectral Stability). For a \( 2\pi/k \)-periodic traveling wave solution \( u(x, t) = w(k(x - ct)) \) of (\ref{y1.3}), where \( w \) and \( c \) are given by (\ref{y2.9}) and (\ref{y2.10}) respectively, the periodic wave is said to be spectrally unstable if the \( L^2(\mathbb{R}) \)-spectrum of the operator \( P[w]\) intersects the open right half-plane of \( \mathbb{C} \). Otherwise, it is spectrally stable.
\end{definition}
 since (\ref{y3.1}) is invariant with respect to the transformations
$$v\mapsto\overline{v}\quad\mathrm{and}\quad\lambda\mapsto\overline{\lambda}$$
and
$$z\mapsto-z\quad\mathrm{and}\quad\lambda\mapsto-\lambda,$$
it follows that the spectrum of \( P[w] \) is symmetric with respect to both the real and imaginary axes. Consequently, the wave \( w \) is spectrally stable if and only if the entire \(L^2(\mathbb{R}) \)-spectrum of \( P[w] \) lies on the imaginary axis. Otherwise, it is spectrally unstable. Since the coefficients of \( P[w] \)  are periodic functions, we use Floquet theory and perform a Bloch wave decomposition. All solutions to (\ref{y3.1}) in \( L^2(\mathbb{R}) \) can be expressed as
\[
v(z) = e^{i\xi z} \phi(z),
\]
where \( \xi \in (-\frac{1}{2}, \frac{1}{2}] \) and \( \phi(z) \) is a \( 2\pi \)-periodic function. In fact,  \( \lambda \in \mathbb{C} \) belongs to the \( L^2(\mathbb{R}) \)-spectrum of \( P[w]  \) if and only if there exists \( \xi \in (-\frac{1}{2}, \frac{1}{2}] \) such that the system
\[
\begin{cases} 
P[w]v = \lambda v \\
v(z + 2\pi) = e^{2\pi i\xi} v(z)
\end{cases}
\]
has a non-trivial solution. Equivalently, this holds if and only if there exist \( \xi \in (-\frac{1}{2}, \frac{1}{2}] \) and a non-trivial \( \phi \in L^2(\mathbb{T}) \) satisfying
\begin{equation}\label{pwt}
    \lambda \phi = e^{-i\xi z} P[w]e^{i\xi z} \phi =: P_\xi[w] \phi.
\end{equation}
Thus, we obtain
\begin{equation}\begin{aligned}
P_{\xi}[w] & =k(1-k^{2}(\partial_{z}+i\xi)^{2})^{-1}\\
 & \quad\big[c(\partial_{z}+i\xi)-ck^{2}(\partial_{z}+i\xi)^{3}-2(1+b)ww_{z}-(1+b)w^{2}(\partial_{z}+i\xi)+\\
 & \quad bk^{2}w_{z}w_{zz}+bk^{2}ww_{zz}(\partial_{z}+i\xi)+bk^{2}ww_{z}(\partial_{z}+i\xi)^{2}+k^{2}w^{2}(\partial_{z}+i\xi)^{3}+2k^{2}ww_{zzz}\big].
\end{aligned}\end{equation}
Moreover, we have the spectral decomposition
\[
\sigma(P[w]) = \bigcup_{\xi \in (-\frac{1}{2}, \frac{1}{2}]} \sigma(P_\xi[w]).
\]
The family of one-parameter operators \(\{P_\xi[w] \}\) defined by \(\xi \in (-\frac{1}{2}, \frac{1}{2}]\) is termed the Bloch operators associated with \( P[w] \), where \(\xi\) is referred to as the Bloch frequency. This enables us to transform the spectral analysis of \(\{P[w] \}\) in \(L^2(\mathbb{R})\) into the study of the spectrum of the Bloch operator family \(\{P_\xi[w] \}\) in \(L^2(\mathbb{T})\).

\begin{lemma} (Symmetry Property). Let \( \xi \in (-\frac{1}{2}, \frac{1}{2}] \). For the Bloch operator \( P_\xi[w] \) acting on the space \( L^2(\mathbb{T}) \), its spectrum satisfies 
\[
\sigma(P_\xi[w] ) =\overline{ \sigma(P_{-\xi}[w] )} = -\sigma(P_{-\xi}[w] ) = -\overline{ \sigma(P_{\xi}[w] )}. 
\]
\end{lemma}

From the above lemma, the spectrum of \( P_\xi[w] \) is symmetric with respect to the imaginary axis. Therefore, it suffices to consider the case \( \xi \in [0, \frac{1}{2}] \). 

\begin{definition}  (Modulational Stability). Suppose that there exists a neighborhood \( B \subset \mathbb{C} \) of the origin and a positive constant \( \xi_0 > 0 \) such that for \( |\xi| < \xi_0 \), the periodic traveling wave solution \( w(z; a,d,k) \) of equation (\ref{y1.3}) satisfies
\[
\sigma(P_\xi[w]) \cap B \subseteq \mathbb{R}i,
\]
then the solution is said to be modulationally stable. Otherwise, it is modulationally unstable.
\end{definition} 

In other words, if \( w \)  maintains spectral stability within a sufficiently small neighborhood around the origin, it is said to possess modulation stability. It should be specifically noted that while modulational instability invariably causes spectral instability, modulational stability does not necessarily indicate spectral stability, as additional instability types may exist in regions far away from the origin.

\section{ Unperturbed operators}

In the following, we determine the modulational stability of specific periodic traveling waves \( w \) by analyzing the spectral properties of the associated Bloch operator \( P_\xi[w] \), where both \( |a| \) and \( |\xi| \) are very small. Based on perturbation theory, we treat the operator \( P_\xi[w] \) as a perturbed form of a constant-coefficient operator \( P_\xi[w_0] \). In fact, we can naturally derive the estimate
\begin{equation}
\|P_\xi[w]-P_\xi[w_0]\|_{H^{1}(\mathbb{T})\to L^{2}(\mathbb{T})}=O(|a|),\quad\mathrm{as}\quad a\to0.
\end{equation}
To study the spectrum of \( P_\xi[w] \), we first need to examine the spectrum of \( P_\xi[w_0] \). When \( a = 0 \) in Lemma 2.1, it corresponds to the trivial wave \( w = w_0 \). A direct Fourier calculation shows that
\begin{equation}\label{y4.2}
P_\xi[w_0]e^{inz}=\lambda e^{inz},\quad n\in\mathbb{Z},
\end{equation}
where
\begin{equation}
\lambda=\frac{i\left(n+\xi\right)\left(\left(n+\xi\right)^2-1\right)k^3\left(c_0-w_0^2\right)}{1+k^2\left(n+\xi\right)^2}=:i\Omega_{n,\xi}.
\end{equation}
We can see that all \( \Omega_{n,\xi} \) are real numbers, indicating that the trivial solution in equation (\ref{y1.3}) is spectrally stable. Furthermore, even when \( a \) is small, individual eigenvalues remain purely imaginary, but multiple eigenvalues colliding on imaginary axis may bifurcate to leave the imaginary axis and trigger instabilities. Consequently, precisely determining the positions of these eigenvalues \( i\Omega_{n,\xi} \), and in particular clarifying their multiplicities, is of vital importance. We summarize the possible collision cases of eigenvalues in the following lemma.

\begin{lemma}
For \( \xi \in [0, 1/2] \), the eigenvalues \( i\Omega_{n,\xi} \) satisfy the following properties:\\
(1) When \( \xi = 0 \), collisions only occur at \( i\Omega_{0,0} = i\Omega_{-1,0} = i\Omega_{1,0} = 0 \);\\
(2) There exists some \( \xi_0 \in (0, 1/2] \) such that collisions may occur between 
\(i\Omega_{-1,\xi_{0}} \) and \(i\Omega_{n,\xi_{0}},n\geq1,\) and between \( i\Omega_{0,\xi_{0}} \) and \(i\Omega_{-n,\xi_{0}},n\geq2\). \\
(3) When \( \xi \in (0, 1/2] \) and \( k^2 < 3 \), no collisions occur.
\end{lemma}

\begin{proof}
(1) When \( \xi = 0 \), it is clear that \( i\Omega_{0,0} = i\Omega_{-1,0} = i\Omega_{1,0} = 0 \). The 
mapping \( n \mapsto \Omega_{n,0} \) is an odd function and strictly increasing  for \( n \geq 1 \). Thus, we have
   \[
   \cdots < \Omega_{-3,0} < \Omega_{-2,0} < \Omega_{-1,0} = \Omega_{0,0} = \Omega_{1,0} = 
   0 < \Omega_{2,0} < \Omega_{3,0} < \cdots.
   \]

(2) For \( \xi \in (0, 1/2] \), noting that the function \( f(x) = \frac{x(x^2 - 1)}{1 + k^2 x^2} \) is odd and monotonically increasing for \( x > \sqrt{3}/3 \). Thus,
$$ \cdots < \Omega_{-3,\xi} < \Omega_{-2,\xi} ,\quad\Omega_{1,\xi} <\Omega_{2,\xi} < \Omega_{3,\xi} < \cdots.$$
Additionally, we have
    
    \begin{equation}\label{-2-xi}
   \Omega_{-2,\xi} = \frac{\left(\xi-2\right)\left(\left(\xi-2\right)^2-1\right)k^3\left(c_0-w_0^2\right)}{1+k^2\left(\xi-2\right)^2}< 0,
    \end{equation}
    \begin{equation}\label{1-xi}
   \Omega_{1,\xi} = \frac{\left(\xi+1\right)\left(\left(\xi+1\right)^2-1\right)k^3\left(c_0-w_0^2\right)}{1+k^2\left(\xi+1\right)^2}> 0,
    \end{equation}
      \begin{equation}\label{-1-xi}
   \Omega_{-1,\xi} = \frac{\left(\xi-1\right)\left(\left(\xi-1\right)^2-1\right)k^3\left(c_0-w_0^2\right)}{1+k^2\left(\xi-1\right)^2}> 0,
    \end{equation}
   and
 \begin{equation}\label{0-xi}
   \Omega_{0,\xi} =  \frac{\xi ( \xi^2-1)k^3\left(c_0-w_0^2\right)}{1 + k^2 \xi^2} < 0.
    \end{equation}
    Thus,
 \begin{equation}
 \cdots < \Omega_{-3,\xi} < \Omega_{-2,\xi}<0<\Omega_{1,\xi} <\Omega_{2,\xi} < \Omega_{3,\xi} < \cdots.
\end{equation}
  There are two possible collision cases: between \( i\Omega_{-1,\xi} \) and \( i\Omega_{n,\xi}, n\geq1,\) and between \( i\Omega_{0,\xi} \) and \( i\Omega_{-n,\xi} , n\geq2.\) 
  
(3) When \( k^2 < 3 \),
   \begin{equation}\label{4.7xz}
   \Omega_{-2,\xi} - \Omega_{0,\xi} = -\frac{2 k^3 (\xi - 1)^2 (3 - 2k^2 \xi + k^2 \xi^2)(c_0-w_0^2)}{\big[k^2 (\xi - 2)^2 + 1\big](k^2 \xi^2 + 1)} < 0
   \end{equation}
   and
   \begin{equation}\label{4.8xz}
   \Omega_{-1,\xi} - \Omega_{1,\xi} = -\frac{2 k^3 \xi^2 (3 - k^2 + k^2 \xi^2)(c_0-w_0^2)}{\big[k^2 (\xi - 1)^2 + 1\big]\big[k^2 (\xi + 1)^2 + 1\big]} <0.
   \end{equation}
By (\ref{-1-xi})-(\ref{4.8xz}), we get
   \[
   \cdots < \Omega_{-3,\xi} < \Omega_{-2,\xi} < \Omega_{0,\xi}< 0 < \Omega_{-1,\xi} < \Omega_{1,\xi}  < \Omega_{-2,\xi} < \Omega_{-3,\xi} < \cdots.
   \]
\end{proof}

Based on the lemma presented above, each of these collision events could cause instability in the eigenvalues of the operator \( P_\xi[w] \). The present study focuses on analyzing the relatively manageable scenario of modulational stability/instability—specifically, investigating eigenvalues in the vicinity of the origin within \( \mathbb{C} \). In particular, when \( k^2 < 3 \), the only potential cause of spectral instability is \( \Omega_{0,\xi} = \Omega_{-1,\xi} = \Omega_{1,\xi} = 0 \). Consequently, spectral stability can be deduced if the criteria for modulational stability are met.

\section{ Analysis of modulational instability}

Based on standard perturbation theory, we treat the operator \( P_\xi[w] \) as a perturbed 
form of the constant-coefficient operator \( P_0[w_0] \). In fact, as $|a|,\xi\to0$, we can obtain
\[
\| P_\xi[w] - P_0[w_0] \|_{H^1(\mathbb{T}) \to L^2(\mathbb{T})} = O(|a| + \xi).
\]
When \( |a| \) and \( \xi \) are small, the spectra of \( P_\xi[w]  \) and \( P_0[w_0] \) remain close to each other.
 From the first  property in Lemma 4.1, we derive the spectral splitting relationship
\[
\sigma(P_\xi[w]) = \sigma_1(P_\xi[w]) \cup \sigma_2(P_\xi[w]), \quad \sigma_1(P_\xi[w]) \cap \sigma_2(P_\xi[w]) = \emptyset.
\]
Here, \( \sigma_1(P_\xi[w]) \) contains three eigenvalues, which are the extensions of the eigenvalues \(i\Omega_{0,\xi},i\Omega_{\pm1,\xi}\) when \( |a| \) is small, denoted as \(\lambda_{j}(\xi,a),j=1,2,3\). The set \( \sigma_2(P_\xi[w]) \) consists of extensions of eigenvalues with \( i\Omega_{n,\xi}, |n| \geq 2 \). When \( |a| \) and \( \xi \) are sufficiently small, 
\( \sigma_2(P_\xi[w]) \) is a subset of the imaginary axis. Thus, we only need to determine the positions of these three eigenvalues in \( \sigma_1(P_\xi[w]) \). Now, we give the main result of the present paper.

\begin{theorem}\label{thm5.1}
 (Modulational Instability). For sufficiently small \( 2\pi/k \)-periodic waves of equation (\ref{y1.3}) with $b>0$, if
\[
g(k,b) < 0,
\]
where
\[
g(k,b)=6+10k^2+26k^4+22k^6+b(12+8k^2+3k^4-11k^6)+b^2(6-2k^2-5k^4+k^6),
\]
then the wave is modulationally unstable. Otherwise, the wave is  modulationally stable. Specifically, if $g(k,b) >0$ and $k^2<3$, the wave is spectrally stable. We give a schematic plot of $g(k,b)<0$ in Figure \ref{fig:instability}.
\end{theorem}
\begin{figure}[htbp]
    \centering
    \subfigure[  The set of $(k,b)$ for which $g(k,b) < 0$. ]{
    \includegraphics[width=0.45\textwidth]{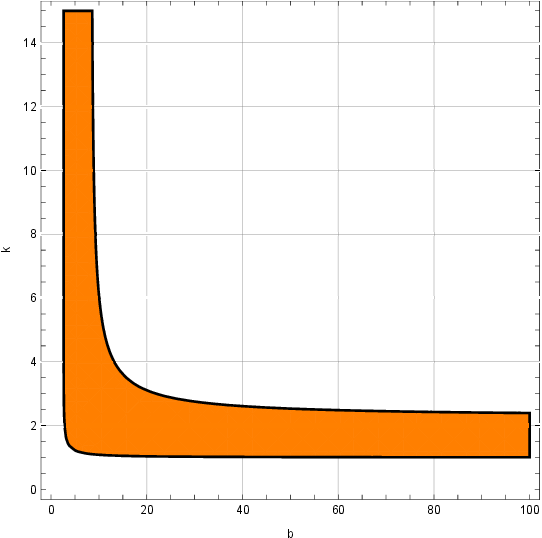}}
    \hspace{0.2in} 
    \subfigure[  A partially enlarged version of the left graph. ]{
    \includegraphics[width=0.45\textwidth]{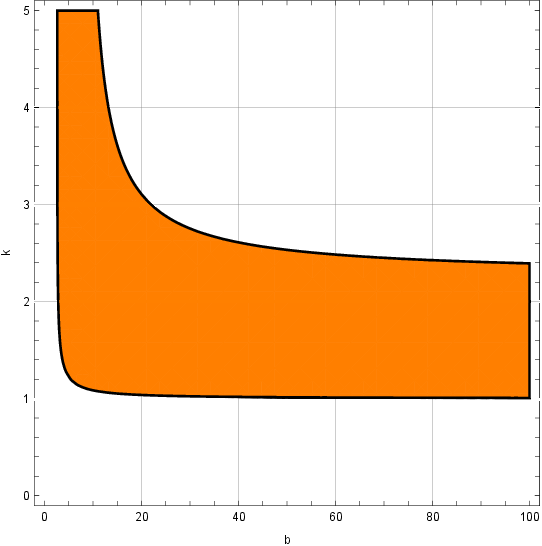}}
    \caption {The colored regions correspond to $g(k,b)<0$, the remaining blank regions correspond to $g(k,b)> 0$,
     the black curves are $g(k,b)=0$.}
    \label{fig:instability}
\end{figure}

\begin{proof}
 To analyze the properties of the eigenvalues  $\lambda_{j}(\xi,a)$ for $j=1,2,3$ when \( |a|,\xi \leq 1 \),
  we aim to project the eigenvalue problem (\ref{pwt}) onto a three-dimensional total eigenspace
 $$\Sigma_{a,\xi}=\bigoplus_{j\in\{1,2,3\}}\ker\left(P_\xi[w]-\lambda_{j}I\right)$$
  Specifically, we first need to construct a suitable basis \( \{ \phi_j(z; a,\xi) \}_{j=1}^3 \) for this space, and then compute the corresponding \( 3 \times 3 \) matrix
\begin{equation}\label{y-eq5.2}
I_a= \left(\frac{\langle\phi_j,\phi_i\rangle}{\langle\phi_i,\phi_i\rangle}\right)_{i,j=1,2,3}
\end{equation}
and
\begin{equation}\label{y-eq5.1}
B_{\xi,a}=\left(\frac{\langle P_\xi[w]\phi_j,\phi_i\rangle}{\langle\phi_i,\phi_i\rangle}\right)_{i,j=1,2,3}=\left(\frac{\langle k\mathcal{L}_{\xi}[w]\phi_{j},\left(1-k^{2}(\partial_{z}+i\xi)^{2}\right)^{-1}\phi_{i}\rangle}{\langle\phi_i,\phi_i\rangle}\right)_{i,j=1,2,3}.
\end{equation}

According to spectral perturbation theory, the critical eigenvalues \(\lambda_{j}(\xi,a)\) can be accurately given by the roots of the characteristic polynomial \( \det(B_{\xi,a} - \lambda I_a) = 0 \). Therefore, we need to find a suitable basis for the eigenspace to compute the above inner products. We first note that from the analysis for \( a = 0 \), the eigenspace $\Sigma_{0,\xi}$ can be spanned for \( |\xi| \ll 1 \) by the \( \xi \)-independent orthogonal basis
\begin{equation}\label{y5.2}
\phi_1(z;\xi,0)=\cos z,\quad\phi_2(z;\xi,0)=\sin z,\quad\phi_3(z;\xi,0)=1.
\end{equation}
Next, we consider the case of \( \xi = 0 \) and \( |a| \ll 1 \), with the aim of constructing a basis for the corresponding total eigenspace.  From previous analysis, this space must be three-dimensional. Differentiating equation (\ref{y2.2}) with respect to \( z \), we find that
\[
P_0[w] w_z = 0.
\]
Differentiating equation (\ref{y2.2}) with respect to \( a \) and \( d \), we obtain
\begin{equation}\label{y5.3}
\mathcal{L}_0[w]w_a=-c_a\left(1-k^2\partial_z^2\right)w_z\quad\mathrm{and}\quad\mathcal{L}_0[w]w_d=-c_d\left(1-k^2\partial_z^2\right)w_z.
\end{equation}
From (\ref{y5.3}), it follows that
$P_0[w]\left(c_aw_d-c_dw_a\right)=0.$
From the above, we derive two linearly independent elements \( w_z \) and \( c_a w_d-c_d w_a  \) for the kernel of \( P_0[w] \). 
Additionally, \( P_0[w] w_d = -k c_d w_z \) gives a generalized eigenvector \( w_d \). Thus, we obtain a basis for \( P_0[w] \):
\begin{equation}\label{y5.3zjd}
\phi_1(z; 0,a) = c_a w_d-c_d w_a, \quad \phi_2(z; 0,a) = w_z, \quad \phi_3(z; 0,a) = w_d.
\end{equation}
By using the expansions of \( w \) and \( c \) from (\ref{y2.9}) and (\ref{y2.10}),
 we   standardize (\ref{y5.3zjd})  as 
\begin{equation}\label{y-eq5.3}
\begin{aligned}
\phi_{1}(z;0,a) & =-\frac{(b+1)d}{2c_{0}}\left(c_{a}w_{d}-c_{d}w_{a}\right)=\cos(z)+a\left(e_{3}+2e_{2}\cos(2z)\right)+\mathcal{O}(a^{2}), \\
\end{aligned}
\end{equation}
\begin{equation}\label{y-eq5.4}
\phi_2(z;0,a):=-\frac{1}{a}w_z=\sin z+2ae_2\sin2z+O(a^2),
\end{equation}
\begin{equation}\label{y-eq5.5}
\phi_{3}(z;0,a){:}=(\partial_dw_0)^{-1}w_d=1+\mathcal{O}(a^2),
\end{equation}
where 
$$e_3:=2e_1-\frac{c_2}{c_0}w_0.$$
 The standard basis (\ref{y-eq5.3})-(\ref{y-eq5.5})  is compatible with the basis when \( a = 0 \).

To summarize, when \( a = 0 \) and \( \xi \) is small, \( P_\xi[w_0] \) possesses three purely imaginary eigenvalues near the origin, with the functions in formula (\ref{y5.2}) constituting an orthonormal basis for the corresponding eigenspace $\Sigma_{0,\xi}$. When \( \xi = 0 \) and \( |a| \) is small, \( P_0[w] \) possesses three eigenvalues at the origin, with the functions in equations (\ref{y-eq5.3})–(\ref{y-eq5.5}) constituting a basis for the eigenspace $\Sigma_{a,0}$. According to perturbation theory,  the functions \( \phi_j(z; 0,a) \) can be extended to a \( \xi \)-dependent basis of the total eigenspace $\Sigma_{a,\xi}$ for \( |a|, \xi \ll 1 \). However, the discussions in references \cite{27} indicate that the variation of the basis functions \( \phi_j(z; \xi,a) \) does not affect subsequent asymptotic computations, as they contribute only higher-order terms beyond the current requirements. Hence, the subsequent calculations shall employ \( \xi \)-independent basis functions \( \phi_j(z; 0,a) \).

For $|a| \ll 1$ and $\xi \ll 1$, a Baker-Campbell-Hausdorff expansion shows that
\begin{equation*}
\mathcal{L}_\xi[w]=\mathcal{L}_0[w]+i\xi L_1-\frac{1}{2}\xi^2L_2+\mathcal{O}(\xi^3)
\end{equation*}
as $\xi \to 0$, where
\begin{equation*}
\begin{aligned}
\mathcal{L}_{0}[w] &= \big[c_{0} - (1+b)w_{0}^{2}\big] \left( \partial_z + \partial_z^3 \right) 
    + a w_{0} \bigg[ 2(k^2 + 1 + b)\sin(z) \\
    &\quad - bk^2 \sin(z)\partial_z^2 - (bk^2 + 2 + 2b)\cos(z)\partial_z 
    + 2k^2 \cos(z)\partial_z^3 \bigg] + \mathcal{O}(a^2),
\end{aligned}
\end{equation*}

\begin{equation*}
\begin{aligned}
{L}_1 : = [\mathcal{L}_0, z] 
              &= \bigl[c_0 - (1+b)w_0^2\bigr]\bigl(1 + 3\partial_z^2\bigr) + a w_0 \bigg[ - (bk^2 + 2 + 2b)\cos(z)\\
              &\quad - 2bk^2\sin(z)\partial_z + 6k^2\cos(z)\partial_z^2 \bigg] + \mathcal{O}(a^2),
\end{aligned}
\end{equation*}
and
\begin{equation*}
\begin{aligned}
L_2 & :=[L_1,z] =6\big[c_{0}-(1+b)w_{0}^{2})\big]\partial_z+\mathcal{O}(a).
\end{aligned}
\end{equation*}
Combining this with the expansions in (\ref{y-eq5.3})-(\ref{y-eq5.5}) yields

\begin{equation*}\begin{aligned}
\mathcal{L}_{\xi}[w]\phi_1& =-2\big[c_{0}-(1+b)w_{0}^{2}\big]i\xi\cos(z)+a\bigg\{12\big[c_{0}-(1+b)w_{0}^{2}\big]e_{2}\sin(2z) \\
 & +2(2k^2+bk^2+2+2b)w_{0}\sin(z)\cos(z)\bigg\} +ia\xi\big[c_{0}-(1+b)w_{0}^{2}\big]\big(e_{3}-22e_{2}\cos(2z)\big)\\
 &+ai\xi w_{0}\bigg[-\left(6k^2+bk^2+2+2b\right)\cos^{2}(z)+2bk^{2}\sin^{2}(z)\bigg]  \\
 &  +3\xi^{2}\bigg[c_{0}-(1+b)w_{0}^{2}\big]\sin(z)+O\left(a^{2}+a\xi^{2}+\xi^{3}\right),
\end{aligned}\end{equation*}

\begin{equation*}\begin{aligned}
\mathcal{L}_{\xi}[w]\phi_2&=-2i\xi\big[c_{0}-(1+b)w_{0}^{2}\big]\sin(z)-3\xi^{2}\big[c_{0}-(1+b)w_{0}^{2}\big]\cos(z)\\
&+a\big[c_{0}-(1+b)w_{0}^{2}\big]\bigg[-12e_{2}\cos(2z)-22i \xi e_{2}\sin(2z)\bigg]\\
 & +aw_{0}\bigg[-\left(2k^2+bk^2+2+2b\right)\cos^{2}(z)+\left(2k^2+bk^2+2+2b\right)\sin^{2}(z)\bigg]\\
 &+ai\xi w_{0}\bigg[-\left(3bk^2+6k^2+2+2b\right)\sin(z)\cos(z)\bigg]  +O\left(a^{2}+a\xi^{2}+\xi^{3}\right),
\end{aligned}\end{equation*}
and
\begin{equation*}\begin{aligned}
\mathcal{L}_{\xi}[w]\phi_{3} & =i\xi\big[c_{0}-(1+b)w_{0}^{2}\big]+aw_{0}\bigg[2(k^2+1+b)\sin(z)\bigg] \\
 & -ai\xi w_{0}\bigg[(bk^2+2+2b)\cos(z)\bigg]+O\left(a^{2}+a\xi^{2}+\xi^{3}\right).
\end{aligned}\end{equation*}
Additionally, noting that for all $ n \in \mathbb{N} $ we have
\begin{equation*}\begin{aligned}
\left(1-k^{2}(\partial_{z}+i\xi)^{2}\right)^{-1}\cos(nz)&=\frac{1+k^{2}(n^{2}+\xi^{2})}{1+k^{4}(n^{2}-\xi^{2})^{2}+2k^{2}(n^{2}+\xi^{2})}\cos(nz)\\
&\quad -i\frac{2k^{2}n\xi}{1+k^{4}(n^{2}-\xi^{2})^{2}+2k^{2}(n^{2}+\xi^{2})}\sin(nz)
\end{aligned}\end{equation*}
and
\begin{equation*}\begin{aligned}
\left(1-k^{2}(\partial_{z}+i\xi)^{2}\right)^{-1}\sin(nz) & =i\frac{2k^{2}n\xi}{1+k^{4}(n^{2}-\xi^{2})^{2}+2k^{2}(n^{2}+\xi^{2})}\cos(nz) \\
 &\quad   +\frac{1+k^{2}(n^{2}+\xi^{2})}{1+k^{4}(n^{2}-\xi^{2})^{2}+2k^{2}(n^{2}+\xi^{2})}\sin(nz).
\end{aligned}\end{equation*}
Combining this with the expansions in (\ref{y-eq5.3})-(\ref{y-eq5.5}) yields
\begin{equation*}\begin{aligned}
(1-k^{2}(\partial_{z}+i\xi)^{2})^{-1}\phi_{1} &  =\frac{1+k^{2}(1+\xi^{2})}{1+k^{4}(1-\xi^{2})^{2}+2k^{2}(1+\xi^{2})}\cos z-i\frac{2k^{2}\xi}{1+k^{4}(1-\xi^{2})^{2}+2k^{2}(1+\xi^{2})}\sin z\\
 & \quad +ae_{3}\frac{1+k^{2}\xi^{2}}{1+k^{4}\xi^{4}+2k^{2}\xi^{2}} +2ae_{2}\bigg(\frac{1+k^{2}(4+\xi^{2})}{1+k^{4}(4-\xi^{2})^{2}+2k^{2}(4+\xi^{2})}\cos2z\\
 & \quad -i\frac{4k^{2}\xi}{1+k^{4}(4-\xi^{2})^{2}+2k^{2}(4+\xi^{2})}\sin2z\bigg)+O\left(a^{2}\right),
\end{aligned}\end{equation*}

\begin{equation*}\begin{aligned}
(1-k^{2}(\partial_{z}+i\xi)^{2})^{-1}\phi_{2}&  =i\frac{2k^{2}\xi}{1+k^{4}(1-\xi^{2})^{2}+2k^{2}(1+\xi^{2})}\cos z+\frac{1+k^{2}(1+\xi^{2})}{1+k^{4}(1-\xi^{2})^{2}+2k^{2}(1+\xi^{2})}\sin z \\
 &\quad +2ae_{2}\bigg(i\frac{4k^{2}\xi}{1+k^{4}(4-\xi^{2})^{2}+2k^{2}(4+\xi^{2})}\cos2z\\
 &\quad+\frac{1+k^{2}(4+\xi^{2})}{1+k^{4}(4-\xi^{2})^{2}+2k^{2}(4+\xi^{2})}\sin2z\bigg)+O\left(a^{2}\right),
\end{aligned}\end{equation*}
and
\begin{equation*}\begin{aligned}
(1-k^{2}(\partial_{z}+i\xi)^{2})^{-1}\phi_{3} & =\frac{1}{1+k^{2}\xi^{2}}+O\left(a^{2}\right).
\end{aligned}\end{equation*}
A direct calculation using the above asymptotic expansions yield
\begin{equation}\begin{aligned}\label{eq:L38z}
\left\langle P_{\xi}[w]\phi_{1},\phi_{1}\right\rangle & =\left\langle k\mathcal{L}_{\xi}[w]\phi_{1},\left(1-k^{2}(\partial_{z}+i\xi)^{2}\right)^{-1}\phi_{1}\right\rangle \\
 & =-\frac{i\xi k\big[c_{0}-(1+b)w_{0}^{2}\big]}{1+k^{2}}+O\left(a^{2}+a\xi^{2}+\xi^{3}\right),
\end{aligned}\end{equation}

\begin{equation}
\left\langle P_{\xi}[w]\phi_{1},\phi_{2}\right\rangle=k\big[c_{0}-(1+b)w_{0}^{2}\big]\left(\frac{3}{2(1+k^{2})}+\frac{2k^{2}}{(1+k^{2})^{2}}\right)\xi^{2}+O\left(a^{2}+a\xi^{2}+\xi^{3}\right),
\end{equation}
\begin{equation}
\left\langle P_{\xi}[w]\phi_{1},\phi_{3}\right\rangle=ai\xi k\bigg\{e_{3}\big[c_{0}-(1+b)w_{0}^{2}\big]+w_0\bigg(\frac{b}{2}k^2-3k^2-1-b\bigg)\bigg\}+O\left(a^{2}+a\xi^{2}+\xi^{3}\right),
\end{equation}
\begin{equation}
\left\langle P_\xi[w]\phi_2,\phi_1\right\rangle=-k\big[c_{0}-(1+b)w_{0}^{2}\big]\left(\frac{3}{2(1+k^2)}+\frac{2k^2}{(1+k^2)^2}\right)\xi^2+O\left(a^2+a\xi^2+\xi^3\right),
\end{equation}
\begin{equation}
\left\langle P_{\xi}[w]\phi_{2},\phi_{2}\right\rangle=-i\xi k\frac{\big[c_{0}-(1+b)w_{0}^{2}\big]}{1+k^{2}}+O\left(a^{2}+a\xi^{2}+\xi^{3}\right),
\end{equation}
\begin{equation}
\left\langle P_{\xi}[w]\phi_{2},\phi_{3}\right\rangle=0+O\left(a^{2}+a\xi^{2}+\xi^{3}\right),
\end{equation}
\begin{equation}\begin{aligned}
\left\langle P_\xi[w]\phi_3,\phi_1\right\rangle=&ai\xi\bigg\{ke_3\big[c_0-(1+b)w_0^2\big]-\frac{2w_0k^3\left(k^2+1+b\right)}{(1+k^2)^2}-\frac{kw_0\left(bk^2+2+2b\right)}{2(1+k^2)}\bigg\}\\
&+O\left(a^2+a\xi^2+\xi^3\right),
\end{aligned}\end{equation}
 \begin{equation}
\left\langle P_\xi[w]\phi_3,\phi_2\right\rangle=\frac{akw_0\left(k^2+1+b\right)}{1+k^2}+O\left(a^2+a\xi^2+\xi^3\right),
 \end{equation}
 \begin{equation}
\left\langle P_{\xi}[w]\phi_{3},\phi_{3}\right\rangle=i\xi k\big[c_0-(1+b)w_0^2\big]+O\left(a^{2}+a\xi^{2}+\xi^{3}\right).
 \end{equation}
Using (\ref{y-eq5.3})–(\ref{y-eq5.5}) and truncating to \( O(a^2) \), we get the following inner products
\begin{flalign}\label{eq:L47z}
\langle \phi_1, \phi_1 \rangle = \langle \phi_2, \phi_2 \rangle = \frac{1}{2}, \quad \langle \phi_1, \phi_2 \rangle = \langle \phi_2, \phi_1 \rangle = \langle \phi_2, \phi_3 \rangle = \langle \phi_3, \phi_2 \rangle = 0,
\end{flalign}
\begin{flalign}\label{eq:L48z}
\langle \phi_1, \phi_3 \rangle = \langle \phi_3, \phi_1 \rangle = ae_3, \quad \langle \phi_3, \phi_3 \rangle = 1
\end{flalign}
as \(|a|\to 0\). Substituting (\ref{eq:L38z})-(\ref{eq:L48z}) into (\ref{y-eq5.2})- (\ref{y-eq5.1}), we obtain
\[
I_a = \begin{pmatrix} 
1 & 0 & 2ae_3\\
0 & 1 & 0 \\
ae_3 & 0 & 1
\end{pmatrix}+O\left(a^{2}\right)
\]
and
$$\begin{aligned}
 B_{\xi,a} &  =i\xi
\begin{pmatrix}
-2k\alpha m_{1} & 0 & 0 \\
0 & -2k\alpha m_{1} & 0 \\
 0 & 0 & k\alpha
\end{pmatrix}+a 
\begin{pmatrix}
0 & 0 & 0 \\
0 & 0 & 2kw_{0}m_{1}\left(k^2+1+b\right) \\
0 & 0 & 0
\end{pmatrix} \\
 &  \quad +ai\xi
\begin{pmatrix}
0 & 0 & \gamma_1 \\
0 & 0 & 0 \\
\gamma_2 & 0 & 0
\end{pmatrix}+\xi^2
\begin{pmatrix}
0 & k\alpha\left(-3m_1+2y_1\right) & 0 \\
k\alpha\left(3m_1-2y_1\right) & 0 & 0 \\
0 & 0 & 0
\end{pmatrix} \\
 & \quad +O\left(a^{2}+a\xi^{2}+\xi^{3}\right),
\end{aligned}$$
where
$$\alpha=c_0-(1+b)w_0^2,\quad y_1=-\frac{2k^2}{\left(k^2+1\right)^2},\quad m_1=\frac{1}{k^2+1},$$
$$\gamma_1=2k\alpha e_3+2ky_1w_0\left(k^2+1+b\right)-kw_0m_1\left(bk^2+2+2b\right),$$
$$\gamma_2=ke_3\alpha+\frac{kw_0\left(bk^2-6k^2-2-2b\right)}{2}$$
as \(\xi,|a|\to 0\). Next, we find the roots of the characteristic polynomial \( \det(B_{\xi,a} - \lambda I_a) = 0 \).

\[
\det(B_{\xi,a} - \lambda I_a) = D_0(\xi,a)+iD_1(\xi,a)\lambda+D_2(\xi,a)\lambda^2+iD_3(\xi,a)\lambda^3.
\]
Let \( D_j = \xi^{3-j}d_j \), \( j = 0, 1, 2, 3 \). Then
\[
  \det{\big(B_{\xi,a}-(i\xi\lambda )I_{a}\big)} =i\xi^{3}\big(d_{3}(\xi,a)\lambda^{3}-d_{2}(\xi,a)\lambda^{2}-d_{1}(\xi,a)\lambda+d_{0}(\xi,a)\big):=i\xi^{3}H(\lambda;\xi,a).
\]
The discriminant of \( H(\lambda;\xi,a) \) is given by
\[
\Delta(a,\xi):= 18d_3d_2d_1d_0 + d_2^2d_1^2 + 4d_2^3d_0 + 4d_3d_1^3 - 27d_3^2d_0^2.
\]
When \( \Delta(a,\xi) > 0 \), the polynomial \( H \) has three real roots, which means the wave is modulationally stable. Conversely,  when \( \Delta(a,\xi) < 0 \), \( H \) has a pair of complex conjugate roots, indicating modulation instability. By direct computation, we have
$$
\Delta(a,\xi)=\Delta(0,\xi)+\Lambda(k,b,d)a^2+O\left(a^2\left(a^2+\xi^2\right)\right)
$$
as \(|a|\to 0\) for \(\xi>0\) small, where
$$
\Delta(0,\xi)=\frac{4b^{\frac{6}{1+b}}d^{\frac{12}{1+b}}k^{18}\left(k^2+3\right)^4\left(7k^2+3\right)^2\xi^2}{\left(k^2+1\right)^{{\frac{14+8b}{1+b}}}}
$$
and
\begin{equation*}\begin{aligned}
\Lambda(k,b,d)=&\left(\frac{2b^{\frac{4-b}{1+b}}d^{\frac{10}{1+b}}k^{14}\left(k^2+3\right)^3\left(7k^2+3\right)}{3\left(k^2+1\right)^{\frac{11+6b}{1+b}}}\right)g(k,b),
\end{aligned}\end{equation*}
with
\begin{equation*}\begin{aligned}
g(k,b)=6+10k^2+26k^4+22k^6+b(12+8k^2+3k^4-11k^6)+b^2(6-2k^2-5k^4+k^6).
\end{aligned}\end{equation*}
Accordingly, the modulational stability or instability behavior is governed by the sign of \(\Lambda(k,b,d)\). Specifically, when \(\Lambda(k,b,d)<0\), it follows that \(\Delta(a,\xi)<0\) for sufficiently small \(\xi>0\) (with the value depending on a sufficiently small fixed \(|a|\)), which is indicative of modulational instability. In contrast, if \(\Lambda(k,b,d)>0\), we have \(\Delta(a,\xi)>0\) for \(\xi>0\) and \(\xi,|a|\) sufficiently small, corresponding to modulational stability. Additionally, it can be readily observed that the sign of \(\Lambda(k,b,d)\) is dictated by the sign of \(g(k,b)\). Recalling Lemma 4.1 and the associated discussion beneath it, we know that for $k^2<3$, the sole form of potential spectral instability is modulational instability. Consequently, such waves are inevitably spectrally stable when 
$g(k,b)>0$. Combining all the above arguments, we thus finish the proof of Theorem 5.1.
\end{proof}

Next, we will analyze the function $g(k,b)$ in detail. For
$$g(k,b)=6+10k^2+26k^4+22k^6+b(12+8k^2+3k^4-11k^6)+b^2(6-2k^2-5k^4+k^6)$$
with $b>0$, set $x=k^2\geq0$. Then $g$ becomes a quadratic in $b$:
$$g(x,b)=A(x)b^2+B(x)b+C(x),$$
where
$$A(x)=x^3-5x^2-2x+6,$$
$$B(x)=-11x^3+3x^2+8x+12,$$
$$C(x)=22x^3+26x^2+10x+6.$$
A simple calculation shows that the equation \(A(x) = 0\) has two positive roots \(x_1=2+\sqrt{10} \approx 5.16228\) and \(x_2=1\). When $0<x<1$ or $x>2+\sqrt{10},$ we have  $ A(x)>0 $, and  $ g(x,b) $ represents a parabola going upwards. When $1<x<2+\sqrt{10}$, it yields that $ A(x)<0,$ and
$ g(x,b) $ is a parabola going downwards. The discriminant is
$$\Delta(x):=B^2(x)-4A(x)C(x)=3x^2\left(11x^4+90x^3+163x^2-120x-96\right).$$
The positive real root of $\Delta(x)$  is $x_0\approx 0.89795$ (i.e., $k_*= \sqrt{x_0} \approx 0.947615$).
 When $x>x_0$, the equation \(g(x,b) = 0\) has two roots
\begin{equation}\label{b1b2}
b_{1}(x)=\frac{- B(x) -\sqrt{\Delta(x)}}{2 A(x)},\quad b_{2}(x)=\frac{- B(x) +\sqrt{\Delta(x)}}{2  A(x)}.
\end{equation}
The sign of \(g(x,b)\) depends on \(x\) and \(b\) as follows.

For \(0 < x < x_0 \approx 0.89795\), we have \(A(x) > 0\) and \(\Delta(x) < 0\), which implies \(g(x,b) > 0\). For \(x_0 \leq  x < 1\), although \(\Delta(x) \geq 0\), both roots of the quadratic equation are negative. Since \(A(x) > 0\), the quadratic opens upward and is positive for all \(b > 0\). Consequently, \(g(x,b) > 0\) holds for all \(0 < x < 1\).

At \(x = 1\), it yields \(A(x) = 0\), so \(g(x,b)\) reduces to \(g(x,b) = B(x)b + C(x)\) with \(B(x) >0\) and \(C > 0\). Solving \(g(x,b) = 0\) gives \(b = -C(x)/B(x) =-\frac{16}{3}\). For all \(b >0\), we have \(g(x,b) > 0\).

For \(1 < x < x_1 \approx 5.16228\), then \(A(x) < 0\), \(\Delta(x) > 0\). The equation \(g(x,b) = 0\) has exactly one positive root \(b_1(x)\) (the other is negative). If \(0 < b < b_1(x)\), then \(g(x,b) > 0\). If \(b > b_1(x)\), then \(g(x,b) < 0\). Moreover, \(b_1(x)\) monotonically decreases with respect to \(x\) in the interval \((1,x_1)\) and satisfies \(\lim_{x \to 1^+} b_1(x) = +\infty\) and \(\lim_{x \to x_1^-} b_1(x) \approx 2.73694\).

At \(x = x_1\), it yields \(A(x) = 0\), so \(g(x,b)\) reduces to \(g(x,b) = B(x)b + C(x)\) with \(B(x) < 0\) and \(C > 0\). Solving \(g(x,b) = 0\) gives \(b = -C(x)/B(x)=\frac{23}{921}(59+16\sqrt{10}) \approx 2.73694\). For \(b < \frac{23}{921}(59+16\sqrt{10})\), \(g(x,b) > 0\). For \(b > \frac{23}{921}(59+16\sqrt{10})\), \(g(x,b) < 0\).

For \(x > x_1\), we have \(A(x) > 0\), \(\Delta(x) > 0\). The equation \(g(x,b) = 0\) has two positive roots \(b_1(x) < b_2(x)\). If \(0 < b < b_1(x)\) or \(b > b_2(x)\), then \(g(x,b) > 0\). If \(b_1(x) < b < b_2(x)\), then \(g(x,b) < 0\). As \(x\) increases from \(x_1\) to \(\infty\), \(b_1(x)\) decreases from \(\frac{23}{921}(59+16\sqrt{10})\approx 2.73694\) to \(\frac{11 - \sqrt{33}}{2} \approx 2.628\), and \(b_2(x)\) decreases from \(+\infty\) to \(\frac{11 + \sqrt{33}}{2} \approx 8.37228\).

 By the above analysis, we can derive the following lemma.
\begin{lemma}
 Given $b_1, b_2$ in (\ref{b1b2}), one has that $ g(k,b)<0 $ holds under the following three cases:

(1) $b > b_1$ for $1<k<\sqrt{2+\sqrt{10}}$; 

(2) $b_{1}<b<b_{2}$ for $k>\sqrt{2+\sqrt{10}}$;

(3) $b >\frac{23}{921}(59+16\sqrt{10})$ for $k=\sqrt{2+\sqrt{10}}.$
\end{lemma}

\begin{remark}
     In \cite{novikov2025}, the authors showed that the small-amplitude periodic traveling
waves of the Novikov Equation are modulational unstable if $k^2>3$.
 In the present paper, taking $b=3$, 
 then $g(k,b)=-2 k^6 - 10 k^4 + 16 k^2 + 96$. By Theorem 5.1, we can derive the conclusion of \cite{novikov2025}.

\end{remark}

\noindent{\bf Acknowledgements}

This work is supported by 
the National Natural Science Foundation of China (No.12571172, No.12475049) 
and by the Key Project of Education Department of Hunan Province (No.24A0667).\\


\begin{thebibliography}{99}

\bibitem{b-novikov2013} Y. Mi, C. Mu, On the Cauchy problem for the modified Novikov equation with peakon solutions, J. Differ. Equ. 254 (2013) 961–982.

\bibitem{novikov2009} V. Novikov, Generalizations of the Camassa–Holm equation, J. Phys. A 42 (2009) 342002.

\bibitem{ch1981} A. Fokas, B. Fuchssteiner, Symplectic structures, their Bäcklund transformation and hereditary symmetries, Phys. D 4 (1981) 47–66.

\bibitem{ch1993} R. Camassa, D. Holm, An integrable shallow water equation with peaked solitons, Phys. Rev. Lett. 71 (1993) 1661–1664.

\bibitem{dp1999} A. Degasperis, M. Procesi, Asymptotic integrability, in: Symmetry and perturbation theory, A. Degasperis, G. Gaeta (Eds.), World Scientific, Singapore, 1999, pp. 23–37.


\bibitem{Constantin2001PRS} A. Constantin, On the scattering problem for the Camassa-Holm equation, Proc. Roy. Soc. London Ser.
 A 457 (2001) 953-970.

\bibitem{Constantin2006IP} A. Constantin, V. Gerdjikov, R. Ivanov, Inverse scattering transform for the Camassa-Holm equation,
 Inverse Problems 22 (2006) 2197–2207.

\bibitem{Constantin2009ARMA}  A. Constantin, D. Lannes, The hydrodynamical relevance of the Camassa-Holm and Degasperis-Procesi equations, 
Arch. Ration. Mech. Anal. 192 (2009) 165–186.

\bibitem{Constantin1998AM} A. Constantin, J. Escher, Wave breaking for nonlinear nonlocal shallow water equations, Acta Math. 181 (1998) 229–243.

\bibitem{Constantin1998ASNSP} A. Constantin, J. Escher, Global existence and blow-up for a shallow water equation,
 Ann. Scuola Norm. Sup. Pisa 26 (1998) 303–328.

\bibitem{Hone etal2009} A. N. W. Hone, H. Lundmark, J. Szmigielski, Explicit multipeakon solutions of Novikov’s cubically nonlinear integrable Camassa-Holm equation, Dyn. Partial Differ. Equ. 6 (2009) 253–289.

\bibitem{A. Himonas and C. Holliman} A. A. Himonas, C. Holliman, The Cauchy Problem for the Novikov equation, Nonlinearity 25 (2012) 449–479.

\bibitem{Grayshan2013} K. Grayshan, Peakon solutions of the Novikov equation and properties of the data-to-solution map, J. Math. Anal. Appl. 397 (2013) 515–521.

\bibitem{yan etal2013} W. Yan, Y. Li, Y. Zhang, The Cauchy problem for the Novikov equation, Nonlinear Differ. Equ. Appl. (2013) 298–318.

\bibitem{Tiglay2011} F. Tiglay, The periodic Cauchy problem for Novikov’s equation, Int. Math. Res. Not. 20 (2011) 4633–4648.

\bibitem{Jiang2012} Z. Jiang, L. Ni, Blow-up phenomena for the integrable Novikov equation, J. Math. Anal. Appl. 385 (2012) 551–558.

\bibitem{wu2011} S. Wu, Z. Yin, Global weak solutions for the Novikov equation, J. Phys. A 44 (2011) 055202.

\bibitem{Palacios2020} J. M. Palacios, Asymptotic stability of peakons for the Novikov equation, J. Differ. Equ. 269 (2020) 7750–7791.


\bibitem{gkbch2013} K. Grayshan, A. A. Himonas, Equations with peakon traveling wave solutions, Adv. Dyn. Syst. Appl. 8 (2013) 217–232.

\bibitem{Zhou S2013} S. Zhou, R. Chen, A few remarks on the generalized Novikov equation, J. Inequal. Appl. 2013 (2013) 560.

\bibitem{Himonas AA2022} A. A. Himonas, C. Holliman, Instability and nonuniqueness for the b-Novikov equation, J. Nonlinear Sci. 32 (2022) 46.

\bibitem{da Silva PL2015} P. L. da Silva, I. L. Freire, On the group analysis of a modified Novikov equation, Springer Proc. Math. Stat. 117 (2015) 161–166.

\bibitem{Efstathiou2022} A. G. Efstathiou, E. N. Petropoulou, Peakon solutions of a b-Novikov equation, Appl. Math. Sci. Eng. 30 (2022) 541–553.

\bibitem{dengxijun2025} X. Deng, S. Lafortune, Spectral instability of peakons for the b-family of Novikov equations, J. Differ. Equ. 415 (2025) 572–588.

\bibitem{fanlili2025} L. Fan, X. Wang, R. Xu, Modulational instability in the b-family equation, Phys. D 481 (2025) 134817.

\bibitem{novikov2025} B. Ehrman, M. A. Johnson, S. Lafortune, Modulational instability of small amplitude periodic traveling waves in the Novikov equation, J. Math. Phys. 66 (2025) 091504.
    
\bibitem{26} M. A. Johnson, J. Oregero, Modulational stability of wave trains in the Camassa-Holm equation, J. Differ. Equ. 446 (2025) 113627.

\bibitem{16} T. B. Benjamin, Instability of periodic wavetrains in nonlinear dispersive systems, Proc. R. Soc. Lond. A 299 (1967) 59–76.

\bibitem{17} T. B. Benjamin, J. E. Feir, The disintegration of wave trains on deep water part 1, Theory, J. Fluid Mech. 27 (1967) 417–430.

\bibitem{18} T. J. Bridges, A. Mielke, A proof of the Benjamin-Feir instability, Arch. Ration. Mech. Anal. 133 (1995) 145–198.

\bibitem{19} H. Q. Nguyen, W. A. Strauss, Proof of modulational instability of Stokes waves in deep water, Comm. Pure Appl. Math. 76 (2023) 1035–1084.

\bibitem{20} M. Berti, A. Maspero, P. Ventura, Full description of Benjamin-Feir instability of Stokes waves in deep water, Invent. Math. 230 (2022) 651–711.

\bibitem{21} M. Berti, A. Maspero, P. Ventura, Benjamin-Feir instability of Stokes waves in finite depth, Arch. Ration. Mech. Anal. 247 (2023) 54.

\bibitem{22} M. Berti, A. Maspero, P. Ventura, Stokes waves at the critical depth are modulationally unstable, Comm. Math. Phys. 405 (2024) 67.

\bibitem{23} E. R. Tracy, H. H. Chen, Y. Lee, Study of quasiperiodic solutions of the nonlinear Schrödinger equation and the nonlinear modulational instability, Phys. Rev. Lett. 53 (1984) 218–221.

\bibitem{24} J. Chen, D. E. Pelinovsky, J. Upsal, Modulational instability of periodic standing waves in the derivative NLS equation, J. Nonlinear Sci. 31 (2021) 32.

\bibitem{25} J. C. Bronski, M. A. Johnson, The modulational instability for a generalized Korteweg-de Vries equation, Arch. Ration. Mech. Anal. 197 (2010) 357–400.

\bibitem{27} V. M. Hur, M. A. Johnson, Modulational instability in the Whitham equation for water waves, Stud. Appl. Math. 134 (2015) 120–143.

\bibitem{28} V. M. Hur, A. K. Pandey, Modulational instability in nonlinear nonlocal equations of regularized long wave type, Phys. D 325 (2016) 98–112.

\bibitem{29} M. Haragus, Stability of periodic waves for the generalized BBM equation, Rev. Roumaine Math. Pures Appl. 53 (2008) 445–463.

\bibitem{31} V. M. Hur, A. K. Pandey, Modulational instability in the full-dispersion Camassa-Holm equation, Proc. A 473 (2017) 20171053.

\bibitem{33} J. Jin, S. Liao, Z. Lin, Nonlinear modulational instability of dispersive PDE models, Arch. Ration. Mech. Anal. 231 (2019) 1487–1530.

\bibitem{34} A. Maspero, A. M. Radakovic, Full description of Benjamin-Feir instability for generalized Korteweg-de Vries equations, SIAM J. Math. Anal. 57 (2025) 3030–3070.

\end{thebibliography}
\end{document}